\documentclass[12pt]{amsart}
\usepackage{amscd,amssymb}
\usepackage[graph,frame,poly,arc]{xy}  
\usepackage[plainpages,backref=page,urlcolor=blue]{hyperref}

\theoremstyle{plain}
\newtheorem{thm}{Theorem}[section]

\newtheorem{prop}[thm]{Proposition}

\theoremstyle{definition}

\theoremstyle{remark}

\newtheorem{rks}[thm]{Remarks}

\numberwithin{equation}{section}
\newcommand{\Z}{\mathbb{Z}}
\newcommand{\Q}{\mathbb{Q}}

\newcommand{\C}{\mathbb{C}}
\newcommand{\PP}{\mathbb{P}}

\DeclareMathOperator{\Hom}{Hom}

\newcommand\enet[1]{\renewcommand\theenumi{#1}
\renewcommand\labelenumi{\theenumi}}

\begin{document}

\title [On fundamental groups of plane curve complements]
{On fundamental groups of plane curve complements}

\author[E. Artal]{Enrique Artal Bartolo$^1$}
\address{Departamento de Matem\'aticas, IUMA\\ 
Universidad de Zaragoza\\ 
C.~Pedro Cerbuna 12\\ 
50009 Zaragoza, Spain} 
\email{artal@unizar.es} 

\author[Alexandru Dimca]{Alexandru Dimca$^2$}
\address{Univ. Nice Sophia Antipolis, CNRS,  LJAD, UMR 7351, 06100 Nice, France. }
\email{dimca@unice.fr}

\thanks{$^1$ Partially supported by
MTM2013-45710-C2-1-P. $^2$ Partially supported by Institut Universitaire de France.}



\begin{abstract} 
In this paper we discuss some properties of fundamental groups and Alexander polynomials of plane curves.
We discuss the relationship of the non-triviality of Alexander polynomials and the notion of (nearly) freeness for irreducible plane curves.
We reprove and restate in modern terms a somewhat forgotten result of Zariski. Finally, we 
describe some topological  properties of curves with abelian fundamental group.
\end{abstract}
 
\maketitle


\section*{Introduction} 
Let $C:f=0$ be a reduced  curve in the complex projective plane $\PP^2$ defined by a homogeneous polynomial $f \in S=\C[x,y,z]$
of degree $d$. The study of the fundamental group of the complement 
$$U=\PP^2 \setminus C$$
goes back to Zariski and has played a central role in the development of Algebraic Topology and Algebraic Geometry.
In this note we consider only some  special topics related to this vast subject. Consider the Milnor fiber associated to the curve $C$, namely the smooth affine surface defined in $\C^3$ by
$$F: f(x,y,z)-1=0,$$
and which is nothing else but a cyclic $d$-fold covering of $U$
(the unique one if $C$ is irreducible). The monodromy transformation $h:F \to F$ given by 
$$h(x,y,z)= \exp(2 \pi i/d) \cdot (x,y,z)$$
corresponds to the natural generator of the deck transformation group of the covering $p:F \to U$,
and induces  monodromy operators $h^*:H^k(F,\C) \to H^k(F,\C)$ for $k=0,1,2$. Note that
$H^k(U, \C)=H^k(F,\C)=0$ for $k>2$ since both $U$ and $F$ are affine varieties of dimension 2.
These monodromy operators are semisimple since $h^d=1_{F}$, and hence one has direct sum decompositions in terms of eigenspaces
$$H^k(F,\C)=\bigoplus_{\lambda\in\mu_d} H^k(F,\C)_{\lambda}$$
where $\mu_d=\{z \in \C: ~ ~ z^d=1\}$. Moreover one has $H^k(F,\C)_1=H^k(U,\C)$, since clearly $U=F/\mu_d$.

In \S\ref{sec:cuspidal} we give an example of a rational cuspidal curve $C$ such that 
the corresponding cohomology group $H^1(F) \ne 0$. This is related to a recent result by U. Walther  \cite{Wa} and the
second named author's
study of nearly free curves in $\PP^2$  \cite{DStNF}.

In \S\ref{sec:zariski} we reprove a classical theorem due to Zariski  \cite{Zar}, 
saying that a cyclic covering of $\PP^2$ of degree $p^a$ with $p$ prime and ramified along an irreducible curve $C$ has zero geometric genus.

In the final section \S\ref{sec:abelian} we consider the properties of the Milnor fiber $F$ for a curve with an abelian $\pi_1(U)$.

The authors 
thank David Massey for useful discussions relating to this subject and Gabriel Sticlaru for help with some computations.

\section{Cuspidal curves}\label{sec:cuspidal} 

In a recent preprint \cite{Wa}, U. Walther has shown 
that we have an { injection}
$$N(f)_{2d-2-j} \to H^2(F, \C)_{\lambda_j},$$
for $j=1,2,...,d$,  where the subscript $\lambda_j$ indicates 
the {eigenspace of the monodromy action} corresponding to the eigenvalue $\lambda_j= \exp(2\pi i (d+1-j)/d)$. 
Here $N(f)=I_f/J_f$, where $J_f$ is the Jacobian ideal of $f$, i.e. 
the ideal in $S$ spanned by the partial derivatives $f_x,f_y,f_z$ and $I_f$ 
is the saturation of the ideal $J_f$ with respect to the ideal $(x,y,z)$.
When $C$ is a rational cuspidal curve, the fundamental group $\pi_1(U)$ can 
be rather complicated, see e.g.~\cite{A,Deg,U}. 
However, at least for rational cuspidal curves  with $d$ even, the second named author and Sticlaru
showed in \cite{DStNF}
that $N(f)$ is rather small, namely $\dim N(f)_k \leq 1$ for all $k$, 
which by definition tells us that $C$ is nearly free, see  \cite{DStNF}. 
In proving this result, it was used the fact that $H^1(F)_{-1}=0$ for any cuspidal curve;
in fact, as it will be stated in \S\ref{sec:zariski} this fact holds for any irreducible
plane curve and for any root of unity whose order is a prime power.
This raises the question whether $H^1(F)=0$ in such cases. We show that this is not the case.

For a plane curve $C:f=0$ we define the $k$-th Alexander polynomial by
$$\Delta_k(C)(t)=\det (tId-h^*|H^k(F,\C)).$$
We say that $\Delta_k(C)$ is trivial, and set $\Delta_k(C)(t)=1$, if $H^k(F,\C)=0$.
\begin{prop}\label{prop:curve}
\label{prop1} Consider the curve of degree $d=6$ given by
$$C: f=(y^2 z-x^3)^2-x^3 y^3.$$
Then the following hold.

\begin{enumerate}
\enet{\rm(\roman{enumi})} 
\item The curve $C$ is rational cuspidal with two  cusps, one is of type $A_2$ (i.e. an ordinary cusp), 
and the other one is a two-Puiseux pair cusp  with Tjurina number 16 and Milnor number 18. 

\item The curve $C$  is nearly free with exponents $d_1=d_2=3$.
\item The curve $C$ has a non-trivial Alexander polynomial $\Delta_1(C)$, more precisely
$$
\Delta_1(C)(t)=t^2-t+1.$$
\end{enumerate}
\end{prop}

Similar examples involving free irreducible curves can also be constructed, see Remark \ref{rk:curve} below.

\begin{proof}

The first two claims follow by a direct computation,
using a computer software as \texttt{Singular} or \texttt{Sagemath} for instance. 
We  recall that the exponents of a nearly free curve $C$ of degree $d$ are the unique pair of positive integers
$d_1 \leq d_2$ such that $d_1+d_2=d$ and $\tau(C)=(d-1)^2-d_1(d_2-1)-1$, where $\tau(C)$ is the total Tjurina number of $C$, see for more details  \cite{DStNF}.

The third claim follows from~\cite{ACO} or \cite{deg:dess} 
where it has been proved that $\pi_1(U)$, $U:=\mathbb{P}^2\setminus C$,
is isomorphic to $(\Z / 2\Z ) * (\Z / 3\Z)$.

For any irreducible curve $C$ one has 
$[\pi_1(U),\pi_1(U)] =\pi_1(F)$, see for instance \cite[Corollary (4.1.10), p. 104]{D1},
where $F$ is the Milnor fiber of~$C$.
The equality $\C^2=H^1(F, \C)$ is well known, see for instance \cite{zr:29} or \cite[Theorem (6.4.9) on p.210]{D1}.
By~\cite{li:82}, see also \cite[Theorem (6.4.13) , p. 215]{D2}, $\Delta_1(t)$ divides the product of the Alexander polynomials of the singular points of~$C$
and all its roots are of order a divisor of~$6$. Recall that the Alexander polynomial of 
the singular point of type~$A_2$ equals $t^2-t+1$; it is not hard to compute (either by hand or using \texttt{Singular})
that the Alexander polynomial
of the cusp with Milnor number~$18$ is 
$$
\frac{(t^{15}+1)(t^6+1)}{(t^2+1)(t+1)}.
$$ 
The result follows immediately.
\end{proof}

\begin{rks}  \label{rk:curve} \mbox{}
\begin{enumerate}
 \item The curve in Proposition~\ref{prop:curve} is part of a family of rational cuspidal curves
with non-trivial Alexander polynomial which are constructed in~\cite{ACO}.
\item Degtyarev~\cite[Theorem B.13]{deg:dess} also computes the fundamental group of the curve in Proposition~\ref{prop:curve}
and shows that it is the only torus curve with a point of multiplicity at least~$4$. 
Note that the curves with equation
$$
(y^2 z-x^3)^2-(x-t y)^3 y^3
$$
have the same topology but distinct projective equivalence types. All of these curves are nearly free with exponents $d_1=d_2=3$.
\item Any other rational cuspidal torus curve must have a non-simple triple point and its fundamental group should coincide
with $\Z/2*\Z/3$, see~\cite[Theorem~7.45]{deg:dess}. It is not hard to check that
only two possible examples exist.
\begin{itemize}
 \item A family of curves with an ordinary cusp and a singularity of type~$E_{18}$:
$$
(xz-y^2)^3-x^2(t(x z-y^2)-x y)^2
$$
The ordinary cuspidal point is $[1:0:0]$ while the other one is $[0:0:1]$; its topological type is the one
of $u^3-v^{10}$, with Milnor number~$18$, but its Tjurina number is~$16$, for $t\neq 0$, and $17$ for $t=0$. For $t \ne 0$ this curve is nearly free with exponents $d_1=d_2=3$, while for $t=0$ we get a free curve with exponents $d_1=2$, $d_2=3$. We  recall that the exponents of a  free curve $C$ of degree $d$ are the unique pair of positive integers
$d_1 \leq d_2$ such that $d_1+d_2=d-1$ and $\tau(C)=(d-1)^2-d_1d_2$.
\item A curve with an $E_6$ and a singularity of type~$E_{14}$:
$$
(xz-y^2)^3-x^2 y^4
$$
The $E_6$ point is $[1:0:0]$ while the $E_{14}$ is $[0:0:1]$; its topological type is the one
of $u^3-v^8$, with Milnor number~$14$, but its Tjurina number is~$13$. This is  a free curve with exponents $d_1=2$, $d_2=3$.
\end{itemize}

\item In~\cite{ACO}, the curve in Proposition~\ref{prop:curve} is used to prove
(without using long or numerical computations) that for a generic torus
curve~$\tilde{C}$ (i.e. with equation $f_2^3+f_3^2=0$, $f_j$ of degree~$j$)
we have that $\pi_1(\tilde{U})$, $\tilde{U}=\PP^2 \setminus\tilde{C}$, 
is isomorphic to $(\Z / 2\Z ) * (\Z / 3\Z)$. This corresponds to the case of a Zariski sextic curve with 6 cusps on a conic.
The key point is the (easy) computation of $\pi_1(U)$ and \cite[Corollary (4.3.2) on p.121]{D1}. 
\end{enumerate}

\end{rks}

\section{On Zariski's result on multiple planes}\label{sec:zariski}

Zariski has shown in \cite{Zar} that a (smooth model of a) multiple plane of degree $p^a$ with $p$ a prime number 
and ramified along an irreducible curve $C \subset \PP^2$ has zero first Betti number and zero geometrical genus. 
In this section we prove the following equivalent form of this statement, combining ideas of Zariski and 
the modern approach involving monodromy and Wang sequences.

\begin{prop}
\label{prop2} Let $C:f=0$ be an irreducible plane curve and $\lambda$ a root of unity of order $p^a$ with $p$ prime and $a \geq 1$.
Then $\lambda$ is not a root of the Alexander polynomial $\Delta_1(C)(t)$.
\end{prop}

\begin{proof}
Let $Y$ be the affine cone over $C$ 
in $\mathbb{C}^3$
and $K =S^5 \cap Y$ be the link of the non-isolated singularity $(Y,0)$.
Then the Milnor fibration $F \to S^5 \setminus K \to S^1$ gives rise to a Wang exact sequence in homology with integer 
coefficients (which are not written in order to keep the notation light)
$$ \cdots \to H_2(S^5 \setminus K) \to H_1(F) \to H_1(F) \to H_1(S^5 \setminus K) \to H_0(F) \to H_0(F) \to \cdots$$
where the morphisms $H_i(F) \to H_i(F)$ are given by $h_*-Id$ for $i=0,1$. 
The Milnor fiber is connected, see~\cite[Proposition (3.2.3) on p. 76]{D1}, and hence 
the last map is
\begin{equation}\label{eq:0}
\mathbb{Z}\equiv  H_0(F)\overset{0}{\longrightarrow} H_0(F) \equiv\mathbb{Z}.
\end{equation}

On the other hand, $H_1(S^5 \setminus K)=H_1(\C^3 \setminus Y)=\Z$, since $C$ is irreducible,
see \cite[Corollary (4.1.4) on p. 103]{D1}. As the morphism $H_1(S^5 \setminus K) \to H_0(F)$ is surjective 
by \eqref{eq:0}, it follows that it is also injective, i.e. the morphism
$$u=h_*-Id:H_1(F) \to H_1(F)$$
is surjective. Let $T \subset H_1(F)$ denote the subgroup of torsion elements and note that $u(T) = T$. 
It follows that $u$ induces 
an isomorphism $v: L \to L$, where $L=H_1(F)/L$ is a free $\Z$-module. 
Note that $H^1(F,\C)=\Hom(H_1(F),\C)=\Hom(L,\C)$ 
If we compute the determinant of $h^*-Id$ with respect to 
a basis $B^*$ we get 
$\pm 1$. It follows that
$$\Delta_1(C)(1)=\pm 1.$$
Since $h^*$ is defined over $\Z$ and has as eigenvalues only $d$-th roots of unity, it follows that the characteristic polynomial $\Delta_1(C)(t)$ is a product of cyclotomic polynomials $\Phi_k$ for $k$ a divisor of $d$, possibly with repeated factors. Then $\lambda$ is  a root of the Alexander polynomial $\Delta_1(C)(t)$ if and only if the factor $\Phi_{p^a}$ is present in the above product.
But it is well known that 
$$\Phi_{p^a}(t)=\frac{t^{p^a}-1}{t^{p^{a-1}}-1}$$
and hence in particular $\Phi_{p^a}(1)=p$. But this is impossible since $p$ does not divide
$\Delta_1(C)(1)=\pm 1.$
\end{proof}

\begin{rks}
\label{rkZariski}
\mbox{}

\begin{enumerate}
\enet{(\roman{enumi})} 
\item Both hypothesis are necessary in the above result. As Zariski has already remarked in \cite{Zar},
the hypothesis that the order of $\lambda$ is a prime power cannot be weakened, e.g. 
the sextic with six cusps on a conic from the previous section has eigenvalues of order 6 for $h^*$ acting on $H^1(F,\C)$. 
The irreducibility of $C$ cannot be dropped either, since the line arrangement
$$(x^3-y^3)(x^3-z^3)(y^3-z^3)=0$$
has eigenvalues of order 3 for $h^*$ acting on $H^1(F,\C)$, see for instance \cite{DProc}.

\item The equivalence of Zariski's result and Proposition \ref{prop2} 
can be seen using the discussion in \cite[p. 206]{D1} about the computation 
of Alexander polynomials using cyclic coverings and the discussion in \cite[p. 218]{D1} about the 
relation between the topology of a normal surface and of its smooth model.

\end{enumerate}

\end{rks}

\section{Curves with abelian \texorpdfstring{$\pi_1(U)$}{pi1(U)}}\label{sec:abelian} 

In this section we consider curves $C:f=0$ such that the fundamental group of the 
complement $\pi_1(U)$ is abelian and see which are the consequences for the topology 
of the corresponding Milnor fiber. Note that $\pi_1(U)$ is abelian when $C$ is a nodal curve, 
see~\cite{De,F}, but also in many other cases as for instance those listed in M. Ulud\u ag's Ph. D. thesis \cite{U}.
There are also some other conditions stated by Nori~\cite{nr:83}.

\begin{prop}
\label{prop3} 
Let $C:f=0$ be a reduced plane curve of degree~$d$ such that $\pi_1(U)$ is abelian. Then the following hold.

\begin{enumerate}
\enet{\rm(\roman{enumi})}
\item\label{prop3i} $\pi_1(F)$ is abelian, more precisely $\pi_1(F)=H_1(F,\Z)=\Z^{r-1}$ , 
where $r$ is the number of irreducible components of $C$. Moreover, $\pi_1(\C^3 \setminus Y)=H_1(\C^3 \setminus Y,\Z)=\Z^r$,
where $Y$ is the affine cone over $C$.

\item\label{prop3ii} If $C$ is irreducible, then the Milnor fiber $F$ is homotopy equivalent to a bouquet of $2$-dimensional spheres $S^2$.

\item\label{prop3iii} The action of the monodromy on $H^1(F,\Z)$ is the identity, hence 
$$\Delta_1(C)(t)=(t-1)^{r-1}.$$

\end{enumerate}
\end{prop}

\begin{proof}
Recall that if $C$ has $r$ irreducible components of degrees $d_1,\dots,d_r$, then
$
\pi_1(U)=H_1(U,\mathbb{Z})
$
is generated by meridians $\mu_1,\dots,\mu_r$, subject to the relation $d_1\mu_1+\dots+d_r\mu_r =0$.
The $d$-fold cyclic covering $p: F \to U$ is determined by the monodromy morphism $\varphi:H_1(U,\mathbb{Z})\to\mathbb{Z}/d$
defined by $\mu_j\mapsto 1\bmod d$. Since there is a  monomorphism $p_{\sharp}: \pi_1(F) \to \pi_1(U)$, we have that
$\pi_1(F)$ is abelian and $\pi_1(F)=H_1(F,\Z)$; moreover, it is isomorphic to $\ker\varphi$ which is clearly
isomorphic to $\mathbb{Z}^{r-1}$.

By \cite[Proposition (3.2.3) on p. 76]{D1}  
we have an exact sequence 
$$ 1 \to [ \pi_1(X), \pi_1(X)] \to \pi_1(F) \to \Z^{r-1} \to 0$$
where $X=S^5 \setminus K$, with $K$ the corresponding link, as in \S\ref{sec:zariski}. On the other hand, we have the obvious $S^1$-bundle
$$ S^1 \to X \to U$$
inducing the exact sequence
$$ 0\to \mathbb{Z}\equiv\pi_1(S^1) \to \pi_1(X) \to \pi_1(U) \to 0,$$
which starts at~$0$ because the first map is injective, using $f:X\to\mathbb{C}^*$.

Since $\pi_1(U)$ is abelian, it follows that  the commutator $H= [ \pi_1(X), \pi_1(X)]$ is a subgroup of $\Z$. 
Assume $H \ne 0$; then $H$ is isomorphic to $\Z$ and the rank of $\pi_1(F)=H_1(F,\Z)$ is $r$, a contradiction. 
Hence, $H=0$ and, since $\C^3 \setminus Y$ is homotopically equivalent to $X$, the proof of the \ref{prop3i} is completed.

Part \ref{prop3ii} follows from the fact that $F$ is a 2-dimensional CW-complex which is simply-connected 
when $C$ is irreducible by~\ref{prop3i}.

The claim that the action of the monodromy is the identity on $H^1(F,\Z)$ is 
equivalent to the claim that this action is the identity on $H^1(F,\Q)$, as no torsion is present. 
But $H^1(F,\Q)$ and the fixed part under the monodromy in $H^1(F,\Q)$, which can be identified to $H^1(U,\Q)$, 
have both dimension $r-1$ over $\Q$, hence they coincide.
\end{proof}

\providecommand{\bysame}{\leavevmode\hbox to3em{\hrulefill}\thinspace}
\providecommand{\MR}{\relax\ifhmode\unskip\space\fi MR }
\providecommand{\MRhref}[2]{%
  \href{http://www.ams.org/mathscinet-getitem?mr=#1}{#2}
}
\providecommand{\href}[2]{#2}

\end{document}